\documentclass{amsart}

\usepackage{amssymb,latexsym,amsmath,epsfig,amsthm}
\usepackage{graphicx}
\usepackage{cancel}
\usepackage{hyperref}
\usepackage{amsfonts}
\usepackage{latexsym}
\usepackage[english]{babel}
\usepackage{graphicx}
\usepackage[T1]{fontenc}

\hypersetup{colorlinks=true}
\theoremstyle{plain}
\newtheorem{theorem}{Theorem}
\newtheorem{cor}{Corollary}

\theoremstyle{definition}
\newtheorem{definition}{Definition}
\newtheorem*{rem}{Remark}
\newtheorem*{example}{Example}

\numberwithin{equation}{section}

\begin{document}

\title{ Further stability results of the  functional equation $f(2x+y)+f\left(\frac{x+y}{2}\right)
=\frac{2f(x)f(y)}{f(x)+f(y)}+\frac{2f(x+y)f(y-x)}{3f(y-x)-f(x+y)}$}
\author{I. Sadani}
\address{Department of Mathematics, University of Mouloud Mammeri, Tizi-Ouzou 15000, Algeria}
\email{sadani.idir@yahoo.fr}

\date{}
\maketitle

\begin{abstract}
    In this paper, we investigate the generalized Hyers-Ulam stability of the following reciprocal functional equation
\begin{equation*}f(2x+y)+f\left(\frac{x+y}{2}\right)
=\frac{2f(x)f(y)}{f(x)+f(y)}+\frac{2f(x+y)f(y-x)}{3f(y-x)-f(x+y)}\end{equation*}
in non-Archimedean space using a direct method. 
\end{abstract}

\section{Introduction}

 Stability of functional equations have been studied by many authors and is
by now a classical subject. It at least goes back to the question posed by
S. M. Ulam \cite{ul} in $1940$. The first affirmative answer to this question was
given by  Hyers \cite{hy} in $1941$. It is noteworthy that the result of Hyers
marked the inauguration of a series of academic publications in the
stability theory of functional equations. In $1978$, Th. M. Rassias \cite%
{rassias-1} gave a remarkable generalization of the Hyers's result which
allows the Cauchy difference to be unbounded.
Latterly, The problem of stability for different functional equations 
and more in different spaces have been considerably explored by many authors (
see for example  \cite{gr}, \cite{be},\cite{ch},%
\cite{ch1},\cite{is},\cite{is1},\cite{ju1},\cite{ju2},\cite{ju3},\cite{ju4},%
\cite{mo},\cite{na},\cite{sa}).

 The functional equation
\begin{equation}  \label{exa}
f(x+y)=\frac{f(x)f(y)}{f(x)+f(y)} \ \ \ (x,y\in\mathbb{R}-\{0\}).
\end{equation} 
is called a  reciprocal functional equation. In particular, every solution of the reciprocal functional equation is said to be a  reciprocal mapping. In $2010$, a generalized Hyers-Ulam stability problem for the reciprocal functional equation was proved by Ravi and Kumar \cite{ku}. 

Since then, numerous  reciprocal functional  equations and reciprocal-quadratic, cubic, quartic, quintic and so on have been discovered and studied by mathematicians with many interesting stability results (see for example \cite{rabian,sadani,ra1,ra2,ra3,ra4, kim} ).

In \cite{sadani}, the author  investigated the general solutions and stability of the following reciprocal type  functional equation 

\begin{equation}\label{eq1}f(2x+y)+f\left(\frac{x+y}{2}\right)
=\frac{2f(x)f(y)}{f(x)+f(y)}+\frac{2f(x+y)f(y-x)}{3f(y-x)-f(x+y)}\end{equation}

with $x,y\in X=\mathbb{R}-\left\{-c\right\}$ where $c$ is the constant appears in the general solution $$f(x)=\frac{a}{x+c}.$$

 In this paper,  we investigate  generalized Hyers-Ulam-Rassias stability of functional equation (\ref{eq1}) in non-Archimedean space.

\section{Preliminaries}
In this section, we recall some basic definitions and properties about non-Archimedean fields.
\begin{definition}
\begin{itemize}
\item A  \textit{valuation} is a function $|\cdot|$ from a field $\mathbb{K}$ into $[0, \infty)$ such that $0$ is the unique element having the $0$ valuation, $|xy|=|x||y|$ and the following strong  triangle inequality holds: 
\begin{equation}\label{arch}
|x+y|\leq\max\{|x|, |y|\}.
\end{equation}
\item A field $\mathbb{K}$ with a valuation on it is called a \textit{valued field}.
\item A valuation which satisfies  the triangle inequality (\ref{arch})  is called a {\it non-Archimedean valuation} and the field is called a {\it non-Archimedean field}.
\end{itemize} 
\end{definition}

\begin{rem} \begin{enumerate}
\item Let $(\mathbb{K}, | \cdot |)$ be a non-Archimedean valued field. Then, for each $n\in\mathbb{N},
|n.1| = |1+1+\cdots+1|\leq\max(|1| , |1| , \cdots , |1|) = 1.$ 
\item Let $\mathbb{K}$ be a field. The map defined by $|x| = 1$ for $x\neq0$ and $|0| = 0$ is
called trivial valuation and is a non-Archimedean valuation.
\end{enumerate}
\end{rem}
\begin{example}[Non-Archimedean fields]
\begin{enumerate}
\item For any prime number $p$ the completion $\mathbb{Q}_p$ of the field $\mathbb{Q}$ of
rational numbers with respect to the $p-$adic valuation $|\cdot|_p$
given by $|a|_p=p^{-r}$ if $a=p^r\frac{m}{n}$ and the integers $m, n$ are
not divided by $p$, is a non-archimedean valued field. Its elements are called $p-$adic numbers. 

\item Let $p$ be an irreducible polynomial in $\mathbb{R}[x]$. The $p-$adic valuation $|\cdot |_p$ on the rational function field 
$\mathbb{R}(x)$ is  defined by $|0|_p = 0$ and $\left|p^u\frac{f}{h}\right|=\frac{1}{e^u}$
where $u\in\mathbb{Z}$, and $f,h\in\mathbb{R}[x],  h\neq0$ are not divisible by $p$. This valuation
is non-Archimedean and $\mathbb{R}(x)$ is a non-Archimedean valued field (for more details, see \cite{pau}).

\end{enumerate}
\end{example}
\begin{definition} Let $Y$ be a vector space over a field $\mathbb{K}$ with a non-Archimedean valuation $|\cdot |$. A function $\left\|\cdot\right\| :  Y\rightarrow[0, \infty)$ is called a {\it non-Archimedean norm} if the following conditions hold:
\begin{enumerate}
\item $ \left\|x\right\|=0$ if and only if $x=0$ for all $x\in Y$, 
\item  $\left\|ax\right\|=|a|\left\|x\right\|$ for all $a\in \mathbb{K}$ and $x\in Y$,\item the strong triangle inequality holds:
$$\left\|x+y\right\|\leq\max\left\{\left\|x\right\|,\left\|y\right\|\right\},$$
for all $x, y\in Y$. Then $\left(Y,\left\|\cdot\right\|\right)$ is called a {\it non-Archimedean normed space}.
\end{enumerate}
\end{definition}
\begin{definition}
 Let $\{x_{n}\}_{n=1}^{\infty}$ be a sequence in a non-Archimedean normed space $Y.$
\begin{enumerate}
 \item A sequence $\{x_{n}\}_{n=1}^{\infty}$ in a non-Archimedean space is a {\it Cauchy sequence} if and only if, the sequence $\{x_{n+1}-x_{n}\}_{n=1}^{\infty}$ converges to zero.
\item The sequence $\{x_{n}\}$ is said to be {\it convergent} if, for any $\epsilon>0$, there are a positive integer $N$ and $x\in Y$ such that
$$
\left\| x_{n}-x\right\|\leq\epsilon,
$$
for all $n\geq N$. Then the point $x\in Y$ is called the {\it limit} of the sequence $\{x_{n}\}$, which is denote by  $\lim_{n\to\infty}x_{n}=x.$
\item If every Cauchy sequence in $Y$ converges, then the non-Archimedean normed space $X$ is called a {\it complete non-Archimedean space}.
\end{enumerate}
\end{definition}

\section{Stability of functional equation (\ref{eq1})}

Hereinafter, we suppose  that $X$  defined as above and $L$ be a complete non-Archimedean field.
 We are now in a position to state our main result.\begin{theorem}\label{theo1}
 Let $\mu :  X^{2}\rightarrow L $  be a mapping such that

\begin{equation}\label{eqt0}\lim_{n\to\infty}|2|^{n}\mu(2^{n+1}x,2^{n+1}y)=0\end{equation}  

 for all $x, y\in X$  and let for each $x\in X$  the limit

\begin{equation}\label{eqt1}\Psi (x)=\lim_{n\rightarrow\infty}\max\{|2|^{k}\mu(0,2^{k+1}x);0\leq k<n\}\end{equation}
 exists. Assume that $f:X\rightarrow L$  is a mapping satisfies
\begin{equation}\label{eqt22}\left\|f(2x+y)+f\left(\frac{x+y}{2}\right)
-\frac{2f(x)f(y)}{f(x)+f(y)}-\frac{2f(x+y)f(y-x)}{3f(y-x)-f(x+y)}\right\| \leq\mu(x,y).
\end{equation}
 Then
\begin{equation}\label{eqt3} 
g(x)=\lim_{n\rightarrow\infty}2^{n}f(2^{n}x)
\end{equation}  

 exists for all $x\in X$  and defines a reciprocal mapping $g : X\rightarrow L$  such that

\begin{equation}\label{eqt4}\left\|f(x)-g(x)\right\|\leq\Psi(x)\end{equation}  
 for all $x\in X$.  Moreover, if

\begin{equation}\label{eqt5} \lim_{j\rightarrow\infty}\lim_{n\rightarrow\infty}\max\{|2|^{k}\mu(0,\ 2^{k+1}x);j\leq k<n+j\}=0\end{equation}
 then $g$  is the unique reciprocal mapping satisfying (\ref{eqt4}).
\end{theorem}

\begin{proof} Substitute $(x,y)$ by $( 0,2x)$ in (\ref{eqt22}), we get

\begin{equation}\label{eqt6}\left\|f\left(x\right)-\frac{2f(0)f(2x)}{f(0)+f(2x)}\right\|\leq\mu(0,2x).
\end{equation}
We replace $x$ by $2^nx$  and multiplying by $2^n$ in (\ref{eqt6}), we obtain

$$\left\|2^nf\left(2^{n}x\right)-\frac{2^{n+1}f(0)f(2^{n+1}x)}{f(0)+f(2^{n+1}x)}\right\|\leq|2|^{n}\mu(0,2^{n+1}x)$$

which is equivalent to
\begin{equation}\label{ineq0}\left|2^nf\left(2^{n}x\right)-\frac{1}{\frac{1}{2^{n+1}f(2^{n+1}x)}+\frac{1}{2^{n+1}f(0)}}\right|\leq|2|^{n}\mu(0,2^{n+1}x)\end{equation}

for all $x\in X$. It follows from (\ref{eqt0}), the right-hand side of (\ref{ineq0}) tends to zero as $n\to\infty$. Thus the sequence $\{2^nf\left(2^{n}x\right)\} $ is a Cauchy sequence. Since $L$ is complete, so $\{2^nf\left(2^{n}x\right)\} $ is convergent. We set $$g(x)=\lim_{n\to\infty}2^nf\left(2^{n}x\right). $$
By induction, it is straightforward to obtain that
\begin{equation}\label{eqt7}\left\|g(x)-f(x)\right\|\leq\max\{|2|^{k}\mu(0,2^{k+1}x);0\leq k<n\}\end{equation}
 for all $n\in \mathbb{N}$ and all $x\in X$. By taking $n$ to approach infinity in (\ref{eqt7}), and using (\ref{eqt1}), we get (\ref{eqt4}). By (\ref{eqt0}) and (\ref{eqt22}), we have \begin{multline}\left\|f(2x+y)+f\left(\frac{x+y}{2}\right)
-\frac{2f(x)f(y)}{f(x)+f(y)}-\frac{2f(x+y)f(y-x)}{3f(y-x)-f(x+y)}\right\|
\\=\lim_{n\to\infty}||2^{n}f(2^{n+1}(2x+y))+2^{n}f\left(\frac{2^{n+1}(x+y)}{2}\right)
-\frac{2^{n+1}f(2^{n+1}x)f(2^{n+1}y)}{f(2^{n+1}x)+f(2^{n+1}y)}\;\\-\frac{2^{n+1}f(2^{n+1}(x+y))f(2^{n+1}(y-x))}{3f(2^{n+1}(y-x))-f(2^{n+1}(x+y))}|| \leq \lim_{n\to\infty} |2|^{n}\mu(2^{n+1}x,2^{n+1}y)=0 .
\end{multline}

for all $x,y\in X$.  Therefore the function $f$ : $X\rightarrow L$ satisfies (\ref{eq1}). To show the uniqueness  of $g$, let $h :X\rightarrow L$ be another function satisfying (\ref{eqt4}). Then
\begin{align*}
\left\| g(x)-h(x)\right\|&=\lim_{n\rightarrow\infty}|2|^{n}\left\| g(2^{n+1}x)-h(2^{n+1}x)\right\| \\
&\leq\lim_{k\rightarrow\infty}|2|^{n}\max\{\left\| g(2^{n+1}x)-h(2^{n+1}x)\right\|, \left\| g(2^{n+1}x)-h(2^{n+1}x)\right\|\}\\
&\leq\lim_{j\rightarrow\infty}\lim_{n\rightarrow\infty}\max\{|2|^{k}\mu(0, 2^{k+1}x);j\leq k<n+j\}\\
&=0
\end{align*}
for all $x\in X$. Therefore $g=h$, and the proof is complete. 
\end{proof}
The following corollaries are an immediate consequence of Theorem \ref{theo1}.
\begin{cor}  Let $\epsilon>0$  be a constant. If $f : X\rightarrow L$  satisfies \begin{equation}\label{eqt2}\left\|f(2x+y)+f\left(\frac{x+y}{2}\right)
-\frac{2f(x)f(y)}{f(x)+f(y)}-\frac{2f(x+y)f(y-x)}{3f(y-x)-f(x+y)}\right\| \leq\epsilon.
\end{equation} 
for all  $x, y\in X$,  then there exists a unique reciprocal mapping $g$ : $X\rightarrow L$  satisfying (\ref{eq1}) and $||f\mathcal{}(u)-g(u)||\leq\epsilon$ {\it for all} $x\in X.$
\end{cor}
\begin{proof}Letting $ \mu(x,y)=\epsilon$ in Theorem \ref{theo1}, we obtain the required result.
\end{proof}
\begin{cor}
 Let $\epsilon\geq 0$ {\it and} $a\neq-1$,  be fixed constants. If $f : X\rightarrow L$  satisfies 

\begin{multline}\label{eqt2}\left\|f(2x+y)+f\left(\frac{x+y}{2}\right)
-\frac{2f(x)f(y)}{f(x)+f(y)}-\frac{2f(x+y)f(y-x)}{3f(y-x)-f(x+y)}\right\|\\ \leq\epsilon(|x|^a+|y|^a).
\end{multline}

  for all $x, y\in X$,  then there exists a unique reciprocal mapping $g:X\rightarrow L$  satisfying (\ref{eq1})  and

$$|f(x)-g(x)|\leq\left\{\begin{array}{ll}
|2|^2\epsilon|x|^{a}, & a>-1,\\
\epsilon\frac{|x|^{a}}{|2|^{a-1}}, & a<-1,
\end{array}\right.$$

{\it for all} $x\in X.$
\end{cor}
\begin{proof} Considering $\mu(x,\ y)=\epsilon(|x|^{a}+|y|^{a})$ in Theorem \ref{theo1} the desired result follows directly.
\end{proof}
\begin{cor} Let $\tau : [0, \infty)\rightarrow]0, +\infty]$  be a function satisfying
$$
\tau(|2|t)\leq\tau(|2|)\tau(t)\ (t\geq 0), \tau(|2|)<|2|.
$$
 Let $\delta>0$  and $f : X\rightarrow L$  be a mapping
 satisfying
   
\begin{multline*}\left\|f(2x+y)+f\left(\frac{x+y}{2}\right)
-\frac{2f(x)f(y)}{f(x)+f(y)}-\frac{2f(x+y)f(y-x)}{3f(y-x)-f(x+y)}\right\|\\ \leq\delta(\tau(|x|)+\tau(|y|))\end{multline*} 
with  $x, y\in X$.
 Then there exists a unique reciprocal mapping $h:X\rightarrow L$  such that
$$
\Vert f(x)-h(x)\Vert\leq\delta\tau(2|x|) \ \  (x\in X).
$$
\end{cor}
\begin{proof}Defining $\mu$ : $ X\times X\rightarrow]0,+\infty]$ by $\mu(x,\ y) =\delta(\tau(|x|)+\tau(|y|))$. On the one hand, we have
$$
 \lim_{n\to\infty}|2|^{n}\mu(2^{n+1}x, 2^{n+1}y) \leq\lim_{n\to\infty}|2|^n\tau(|2|)^{n+1}\mu(x, y)=\lim_{n\to\infty}(|2|^{}\tau(|2|))^{n+1}\mu(x,\ y)
$$

$$=\lim_{n\to\infty}\tau(|2|)(|2|\tau(|2|))^{n}\mu(x, y)=0.$$
for all $x, y\in X$. On the other hand,
$$\Psi (x)=\lim_{n\rightarrow\infty}\max\{|2|^{k}\mu(0,2^{k+1}x);0\leq k<n\}=\mu(0,2x)=\delta(\tau(|2x|)),$$
and
$$\lim_{j\rightarrow\infty}\lim_{n\rightarrow\infty}\max\{|2|^{k}\mu(0,\ 2^{k+1}x);j\leq k<n+j\}= \lim_{j\to\infty}|2|^{j}\mu(0,\ 2^{j+1}x)=0.$$
Applying Theorem \ref{theo1}, we get desired result.
\end{proof}
 
\section{Conclusion}
We have proved the stability of  a reciprocal type functional equation
\begin{equation*}f(2x+y)+f\left(\frac{x+y}{2}\right)
=\frac{2f(x)f(y)}{f(x)+f(y)}+\frac{2f(x+y)f(y-x)}{3f(y-x)-f(x+y)}\end{equation*}
in non-Archimedean space using the direct method.


\end{document}